\documentclass[final,twosides,a4paper,11pt]{amsart}

\usepackage{latexsym,amsfonts,amssymb,amsthm,amsmath,amscd,eufrak,mathrsfs}
\usepackage{fancyheadings}
\usepackage{fancybox,shadow, color,float}
\usepackage{graphicx}

\theoremstyle{plain}
   \newtheorem{theorem}{Theorem}
   \newtheorem{corollary}[theorem]{Corollary}
   \newtheorem{lemma}[theorem]{Lemma}
   \newtheorem{proposition}[theorem]{Proposition}
   
\theoremstyle{definition}
   
\theoremstyle{remark}

\newcommand{\diam}{\mathrm{diam}}
\def\RR{\mathbb{R}}
\def\NN{\mathbb{N}}

\numberwithin{equation}{section}

\allowdisplaybreaks

\begin{document}

\title[Powers of distances as  Muckenhoupt weights]
 {Powers of distances to lower dimensional sets as  Muckenhoupt weights}

\author[H. Aimar]{Hugo Aimar}
\address{Instituto de Matem\'atica Aplicada del Litoral
(CONICET-UNL), Departamento de Matem\'{a}tica (FIQ-UNL),  Santa Fe,
Argentina.} \email{haimar@santafe-conicet.gov.ar}

\author[M. Carena]{Marilina Carena}
\address{Instituto de Matem\'atica Aplicada del Litoral
(CONICET-UNL), Departamento de Matem\'{a}tica (FHUC-UNL),  Santa Fe,
Argentina.} \email{mcarena@santafe-conicet.gov.ar}

\author[R. Dur\'an]{Ricardo Dur\'an}
\address{Instituto de Investigaciones Matem\'aticas ``Luis A. Santal\'o''
(CONICET-UBA), Departamento de Matem\'atica (UBA), Buenos Aires,
Argentina.} \email{rduran@dm.uba.ar}

\author[M. Toschi]{Marisa Toschi}
\address{Instituto de Matem\'atica Aplicada del Litoral
(CONICET-UNL), Departamento de Matem\'atica (FIQ-UNL), Santa Fe,
Argentina.} \email{mtoschi@santafe-conicet.gov.ar}


\subjclass[2010]{Primary 28A25; Secondary 28A78}

\keywords{Ahlfors spaces, Hardy-Littlewood maximal operator, Muckenhoupt weights, Hausdorff measure}

\maketitle

\begin{abstract}
Let $(X,d,\mu)$ be an Ahlfors metric measure space. We give sufficient conditions on a closed set $F\subseteq X$ and on a real number $\beta$
in such a way that $d(x,F)^\beta$ becomes a Muckenhoupt weight.
 We give also some illustrations to regularity of solutions of partial differential equations and regarding some classical fractals.
\end{abstract}

\section{Introduction}

Under different conditions for a domain $\Omega$, weighted norm estimates for solutions of linear and nonlinear equations have
 been studied by several authors (see
for example \cite{D-S}, \cite{DuranSanmartinoToschi}, \cite{DuranSanmartinoToschi2}, \cite{PS}), where the weight of interest is a power of the
 distance to the boundary $\partial\Omega$.

The class of Muckenhoupt weights   $A_p(\RR^n)$ is a fundamental tool in real and harmonic analysis. The first non-trivial examples of weights in $A_p(\RR^n)$ are the weights
$|x|^\beta$ for $-n<\beta< n(p-1)$.
The results in \cite{DuranSanmartinoToschi} show that in a domain
whose boundary has dimension $n-1$, if the domain is smooth enough
then $d^\beta(x,\partial\Omega)\in A_p(\RR^n)$ for $-1<\beta< (p-1)$.
In \cite{DuranLopez}  this result has been generalized to some $s$-dimensional compact sets  $F$ in $\RR^n$ with $0\le s<n$. They proved
that $d^\beta(x,F)\in A_p(\RR^n)$ for $-(n-s)<\beta< (n-s)(p-1)$.

Therefore thinking the domains on $\RR^n$ as complements of closed sets, it seems natural to try to consider the somehow heterogeneous situation in which the boundary shows different dimensions at different points. With that generality the problem looks hard. But with some extra hypotheses on the structure of $\partial\Omega$ some extensions are possible.
The techniques used here extend naturally to general metric measure space $(X,d,\mu)$ satisfying the so called Ahlfors condition, which is a particular case of space of homogeneous type.

The paper is organized as follows.
In Section~\ref{section: general} we give some definitions and notation on metric spaces and the Hardy-Littlewood maximal function of measures obtaining the
finiteness of such function for certain measures.
The main results of this note are contained in Section~\ref{section: main results}. There we prove that powers of the Hardy-Littlewood maximal function of measures belong to the
Muckenhoupt class and in some cases we can describe the behavior of such function giving a family of weights on metric measure spaces.
In Section~\ref{section: applications} we give some applications. First we obtain weighted bounded estimates for gradients of solutions of polyharmonic equations and then we produce families of weights on some classical fractals.

\section{Hardy-Littlewood maximal functions of measures}\label{section: general}
 Let $X$ be  a  set.  A \emph{\textbf{quasi-distance}} on $X$ is  a non-negative
symmetric function  $d$ defined on $X\times X$ such that
$d(x,y)=0$ if and only if $x=y$, and there exists a constant
$K\geq 1$ such that  the inequality \[d(x,y)\leq K(d(x,z)+d(z,y))\] holds for every $x,y,z\in X$.
We will refer to $K$ as the triangle constant for $d$.
A quasi-distance $d$ on $X$ induces a topology  through the
neighborhood system given by the family of all subsets of $X$
containing a $d$-ball $B(x,r)=\{y\in X:d(x,y)<r\}$,
$r>0$ (see \cite{CoifWeiss}). In a quasi-metric space $(X,d)$ the \emph{\textbf{diameter}} of
a subset $E$ is defined as
\[\diam(E)=\sup\{d(x,y):x,y\in E\},\]
and the distance between a point $x\in X$ and a set $E$ is defined by $d(x,E)=\inf\{d(x,y): y\in E\}$.

Throughout this paper $(X,d)$ shall be a  quasi-metric
space such that the $d$-balls are open sets.

 We shall say that $(X,d,\mu)$ is a \emph{\textbf{space of homogeneous type}} if  $\mu$ is a non-negative  Borel measure $\mu$ satisfying the \emph{\textbf{doubling condition }}
\[0<\mu(B(x,2 r))\leq A\mu(B(x,r))<\infty\]
for some constant $A$, for every $x\in X$ and every $r>0$. If  we also have $\mu(\{x\})=0$ for every $x\in X$ then we say that $(X,d,\mu)$ is a \emph{\textbf{non-atomic}} space of homogeneous type. 

We shall now recall a basic
 property of spaces of homogeneous type
that we shall need. This property is actually contained in \cite{CoifWeiss}, and reflects the fact that spaces of homogeneous
type have finite  metric (or Assouad)  dimension (see \cite{Asso}).
The expression \emph{\textbf{finite metric dimension}} means that there exists a constant $N\in\mathbb{N}$
  such that no ball of
radius \mbox{$2r$} contains more than $N$ points of any $r$-disperse
subset of $X$. A set $U$ is said to be  \emph{$\boldsymbol{r}$\textbf{-disperse}}
if $d(x,y)\geq r$ for every $x,y \in U$, $x\neq y$.  An \emph{$\boldsymbol{r}$\textbf{-net}} is  a maximal $r$-disperse set. It is easy to check that
 $U$ is an $r$-net in $X$ if and only if $U$ is an $r$-disperse and  $r$-dense set in $X$, where \emph{$\boldsymbol{r}$\textbf{-dense}} means that for every $x\in X$ there exists $u\in U$ with $d(x,u)<r$.
It is well known that  if a quasi-metric space $(X,d)$ has finite metric dimension, then every bounded subset $F$ of $X$ is totally bounded, so that for every  $r>0$ there exists a finite $r$-net on $F$, whose cardinal depends on $\diam(F)$ and on $r$.

On the other hand, every compact quasi-metric space with finite metric dimension carries a nontrivial doubling measure (see \cite{Wu} or \cite{VK87}).

Let $(X,d,\mu)$ be a  space of homogeneous type.
For a given locally integrable function $f$, the \emph{\textbf{Hardy-Littlewood maximal operator}} is given by
\[\mathcal M_\mu f(x)=\sup \frac{1}{\mu(B)}\int_{B} |f|\,d\mu,\]
where the supremum is taken over the family of the $d$-balls $B$ containing $x$.
Since $\mu$ is doubling then $\|\mathcal M_\mu f\|_{L^p(d\mu)}\leq C_p \|f\|_{L^p(d\mu)}$ for $1<p<\infty$
and $\mu\left(\left\{\mathcal M_\mu f>\lambda\right\}\right)\leq\frac{C}{\lambda}\|f\|_{L^1(d\mu)}$ (see \cite{CoifWeiss}).

The definition of the Hardy-Littlewood maximal operator can be extended to a non-negative Borel measure $\nu$ such that every ball has finite $\nu$-measure
by
\[\mathcal M_\mu\nu(x)= \sup \frac{\nu(B)}{\mu(B)},\]
where the supremum is taken over the family of the $d$-balls $B$ containing $x$. Since $\mu$ is doubling, $\mathcal M_\mu\nu(x)$ is equivalent to its centered version, i.e. \[M_\mu\nu(x)= \sup_{r>0}\frac{\nu(B(x,r))}{\mu(B(x,r))}.\]

When $\nu$ is a finite measure the analogous of the weak type inequality given above shows that $M_\mu \nu$ is finite $\mu$-almost everywhere. We aim to give sufficient conditions other than the finiteness of $\nu$ in order to have the $\mu$-almost everywhere finiteness of $M_\mu \nu$.

We shall deal with spaces of finite Hausdorff dimension. Some comments regarding the terminology are in order.
In the bibliography belonging to geometric measure theory, such as \cite{Falconer}, an $s$-set $E$ is one for which $0< \mathscr H^s(E)<\infty$ where $\mathscr H^s$ is the Hausdorff measure of dimension $s$. In some references related to problems of harmonic analysis and partial differential equations, see for example \cite{Sjoding}, the expression $s$-set is used to name a set that supports a measure $\nu$ for which $\nu(B(x,r))$ behaves as $r^s$ for $r$ small. This condition implies the above one.
On the other hand when dealing with operators such as the Hardy-Littlewood maximal, the global behavior of the given set, aside its local behavior, becomes relevant. In this direction, again, two different terminologies appear in the literature. When there exists a measure $\nu$ such that $\nu(B(x,r))$ behaves as $r^s$ for $r$ up to the diameter of the support of $\nu$, which could be unbounded, the space is said to be $s$-Ahlfors. See for example \cite{Sjoding} and \cite{gromov}.
In \cite{M-S} for a given space of homogeneous type  $(X,d,\mu)$, a quasi-distance $\delta$ can be constructed on $X$ satisfying that the measure of a $\delta$-ball centered at $x$ with radius $r$, behaves as $r$ when $\mu(\{x\})=0$ if $r$ is less than $\mu(X)$. This property is named normality.

For the sake of simplicity we shall adopt along this note the following definitions.

Assume that $(X,d)$ is a fixed quasi-metric space. A subspace $(Y,d)$ of $(X,d)$ is said to be \emph{$\boldsymbol\alpha$\textbf{-Ahlfors with measure} $\boldsymbol\mu$} if $\mu$ is a Borel measure supported in $Y$ such that there exists a constant $c\ge 1$ satisfying the inequalities
 \begin{equation}\label{eq: normal}
c^{-1}r^\alpha\leq \mu(B(x,r))\leq cr^\alpha,
 \end{equation}
for every $x\in Y$ and every $0<r< \diam(Y)$.
 It easy to see that each  $\alpha$-Ahlfors space with measure $\mu$  is  a non-atomic space of
homogeneous type with doubling constant for $\mu$ which  only depends  on
$c$ and $\alpha$.
We will refer to the triangle constant
$K$ and the constants $c$ and $\alpha$ in \eqref{eq: normal} as the \emph{\textbf{geometric constants}} of the space.
 When \eqref{eq: normal} holds only for $0<r<r_0$ for some positive $r_0$, we say that $(Y,d)$ is \emph{\textbf{locally $\boldsymbol\alpha$-Ahlfors with measure $\boldsymbol\mu$}}. The constant $c$ in (\ref{eq: normal}) is said to be  \emph{\textbf{a constant for the Ahlfors condition of} $\boldsymbol\mu$}.

The next proposition shows that both concepts coincide when the space is bounded.

\begin{proposition}
Let  $F$ be a closed and bounded subset of a quasi-metric space $(X,d)$.
If $(F,d)$ is locally $s$-Ahlfors with measure $\nu$, then we have that  $(F,d)$ is
$s$-Ahlfors with measure $\nu$.
\end{proposition}

\begin{proof} If $\diam(F)\leq r_0$ the result is trivial, so that we shall assume that $\diam(F)>r_0$. Take $x\in F$. If $0<r<r_0$ we use that  $(F,d)$ is locally $s$-Ahlfors with measure $\nu$, so that only we need to consider the case $r_0\leq r<\diam(F)$. Being $F$ a bounded set, since $(X,d)$ has finite metric dimension, there exists a finite $\frac{r_0}{2}$-net $U$ in $F$, let us say $U=\{x_1,\dots, x_I\}$, where $I$ depends only on $\diam(F)$  and $r_0$. Then
\[
  \nu(B(x,r))\leq \sum_{i=1}^I \nu\left(B\left(x_i,\frac{r_0}{2}\right)\right)\leq I c 2^{-s}r_0^{s}< Ic r^{s}.
\]
\end{proof}

The main result of this section, which we shall use in order to apply Theorem~\ref{teo:maximal en A1}  to build Muckenhoupt weights in the next section, is contained in the following proposition.

\begin{proposition}\label{propo: s-sets miden cero}
Let  $(X,d)$ be  $\alpha$-Ahlfors with measure $\mu$ and $F$ closed in $X$. If $(F,d)$ is  $s$-Ahlfors with measure $\nu$, where $0\leq s<\alpha$, then
 $M_\mu\nu(x)<\infty$ for $\mu$-almost every $x\in X$.
\end{proposition}

\begin{proof}
If $x\notin F$ and $r\leq d(x,F)$, then $B(x,r)\cap F=\emptyset$ and hence we have $\nu(B(x,r))=0$. So that, in this case, we can assume $r>d(x,F)>0$. Fix $y\in F$ such that $d(x,y)<\frac32 d(x,F)$. Then  $B(x,r)\subseteq B(y,3Kr)$ and
we obtain
\[\frac{\nu(B(x,r))}{\mu(B(x,r))}\leq \frac{\nu(B(y,3Kr))}{\mu(B(x,r))}\leq  c\tilde c(3K)^s r^{s-\alpha}<C d(x,F)^{s-\alpha},\]
with $C=c\tilde c(3K)^s$, where  $K$ denotes the triangle constant for $d$, and
$c$ and $\tilde c$ are constants for the Ahlfors condition of $\mu$ and $\nu$ respectively. In other words, $C$ depends only on the geometric constants of $(X,d)$ and $(F,d)$.
Hence $M_\mu\nu(x)<\infty$ for $x\notin F$.

So we have to prove that $\mu(F)=0$.
Fix $x_0\in F$ and for each natural number $n$ let $F_n=F\cap B(x_0,n)$. Then $\mu(F)\leq \sum_{n=1}^\infty\mu(F_n)$, so that it is enough to prove that
$\mu(F_n)=0$ for every $n$. To see this, fix $n\in\mathbb N$ and $0<\rho <n$.
Since $(X,d)$ has finite metric dimension and $F_n$ is  bounded, there exists a finite $\rho$-net $U$ in $F_n$, let us say $U=\{x_1,x_2,\dots,x_{I_{\rho}}\}$.  Hence $\{B(x_i,\rho): i=1,\dots,I_{\rho}\}$ is a cover of $F_n$, so that $\mu(F_n)\leq \sum_{i=1}^{I_{\rho}}\mu(B(x_i,\rho))\leq c\rho^\alpha I_{\rho}$. To estimate $I_{\rho}$, notice that $B\left(x_i,\frac{\rho}{2K}\right)\cap B\left(x_j,\frac{\rho}{2K}\right)=\emptyset$ for $i\neq j$. Then
\[I_{\rho}\frac{\rho^s}{\tilde c (2K)^s} \le \sum_{i=1}^{I_{\rho}}\nu\left(B\left(x_i,\frac{\rho}{2K}\right)\right)=
\nu\left(\bigcup_{i=1}^{I_{\rho}}B\left(x_i,\frac{\rho}{2K}\right)\right)\leq \nu(B(x_0,2Kn)).\]
So that $I_{\rho}\leq \nu(B(x_0,2Kn))\tilde c (2K)^s\rho^{-s}$  and we have that
\[\mu(F_n)\leq c\tilde c \nu(B(x_0,2Kn))(2K)^{\alpha+s} \rho^{\alpha-s}.\] Taking $\rho\to 0$ we obtain $\mu(F_n)=0$.
\end{proof}

Notice that
if $x\in F$, then we have that
\[\frac{\nu(B(x,r))}{\mu(B(x,r))}\geq \frac{1}{c\tilde c}r^{s-\alpha},\]
and taking supremum over $r>0$ we obtain $M_\mu\nu(x)=\infty$.

\section{Main results}\label{section: main results}

The theory developed by Muckenhoupt in \cite{Muck} provides necessary and sufficient conditions on a weight  $w$ defined on $(\RR^n,|\cdot|,\lambda)$ in order to obtain weighted estimates for the maximal operator.
These functions $w$ are known as $A_p$-Muckenhoupt weights.

Let $(X,d,\mu)$ be a quasi-metric measure space such that every $d$-ball has positive and finite $\mu$-measure. A \textbf{\emph{weight}} $w$  on $X$ is a locally integrable non-negative function defined on $X$. By locally we mean integrable over balls, i.e. $\int_B w\,d\mu<\infty$ for every $d$-ball $B$ in $X$.

For $1<p<\infty$ the \emph{\textbf{Muckenhoupt class $\boldsymbol{A_p(X,d,\mu)}$}} is defined as the set of all weights $w$ defined on $X$ for which there exists a constant $C$ such that the inequality
\[\left(\frac{1}{\mu(B)}\int_B w\,d\mu\right)\left(\frac{1}{\mu(B)}\int_B w^{-\frac{1}{p-1}}\,d\mu\right)^{p-1}
\leq C\]
holds for every $d$-ball $B$ in $X$. For $p=1$, we say that $w\in A_1(X, d,\mu)$ if there exists a constant $C$ such that
 \begin{equation*}
  \frac{1}{\mu(B)}\int_B w\, d\mu\le C\, w(x)
 \end{equation*}
holds for every $d$-ball $B$ in $X$ and $\mu$-almost every $x\in B$. Set $A_\infty(X,d,\mu)=\bigcup_{p\geq 1} A_p(X,d,\mu)$.

\smallskip
 It is a well known  result in the theory of Muckenhoupt weights that if $w\in A_p(X,d,\mu)$, then $wd\mu$ is doubling on $X$ provided that $(X,d,\mu)$ is a space of homogeneous type. The classical reference for the basic theory of Muckenhoupt weights is Chapter~IV in the book~\cite{garcia-rubio}. A celebrated result of the theory proved in $\mathbb R^n$ by P. Jones in \cite{jones}, which extends to space of homogeneous type, is the factorization theorem: every $A_p$ weight $w$ can be written as $w=w_0\, w_1^{1-p}$ with $w_0$ and $w_1$ in $A_1$. This is an important result of the Muckenhoupt weights and it is known as factorization property.

The basic general result which shall be useful for the construction of Muckenhoupt weights as powers of the distance to some particular sets in $X$ is contained in the following statement.

\begin{theorem}\label{teo:maximal en A1}
Let $(X,d,\mu)$ be a space of homogeneous type. Let $\nu$ be a Borel measure
such that
$M_\mu\nu(x)<\infty$ for $\mu$-almost every $x\in X$.   Then  $(M_\mu\nu)^\gamma\in A_1(X,d,\mu)$ for every $0\le \gamma <1$.
\end{theorem}

The proof of Theorem~\ref{teo:maximal en A1} is based in Kolmogorov inequality. We will follows the lines that in $\mathbb R^n$ for $\nu$ absolutely continuous are given in  \cite{Javi} and for general $\nu$ in $\mathbb R^n$ in   \cite{garcia-rubio}.
However, non trivial technical modification are needed to extend the result to a general space of homogeneous type.

\begin{proof}[Proof of Theorem~\ref{teo:maximal en A1}]
We have to prove that there exists a constant $C$ such that the inequality
\begin{equation}\label{a_1}
\frac{1}{\mu(B)}\int_{B} (M_\mu \nu)^\gamma\, d\mu\le C\, M_\mu \nu(x)^\gamma
\end{equation}
holds for every $d$-ball $B$ in $X$ and $\mu$-almost every $x\in B$.
 Let us
 fix a $d$-ball $B=B(x_0,r_0)$ and  write $\nu=\nu_1+\nu_2$, where $\nu_1$ is the restriction to  $B(x_0,2Kr_0)$  of $\nu$ and $\nu_2=\nu-\nu_1$. Then for $0\le \gamma <1$ we have
\[(M_\mu\nu)^\gamma\le (M_\mu\nu_1)^\gamma + (M_\mu\nu_2)^\gamma,\] and it is enough to prove  (\ref{a_1}) with  $\nu_1$ and $\nu_2$ on the left hand side.

Since  $M_\mu$ is of weak type $(1,1)$ and $\nu_1$ is finite, we can apply the Kolmogorov's inequality  on $(X,d, \mu)$, namely
\[
\int_{E} (M_\mu\nu_1)^\gamma\, d\mu \leq  \mu(E)^{1-\gamma} \nu_1(E)^\gamma=\mu(E)^{1-\gamma} \nu(E\cap B(x_0,2Kr_0) )^\gamma
\]
for every measurable subset $E$ of $X$ with finite measure,
  see for example \cite{DeGuz}, and obtain a constant $C$ depending on $\gamma$ such that
\begin{align*}
 \frac{1}{\mu(B)}\int_{B} (M_\mu\nu_1)^\gamma\, d\mu &
\le C \left(\frac{\nu(B)}{\mu(B)} \right)^\gamma \le C M_\mu\nu(x)^\gamma.
\end{align*}

In order to analyze $M_\mu \nu_2$, note that if $y\in B$ and $B_1=B(x_1,r_1)$ is any $d$-ball containing $y$ with
$\nu_2 (B_1)>0$,  then $r_1>\frac1{2K} r_0$.
In fact, for $z\in B_1\cap (X\setminus B(x_0,2Kr_0))$ we have
\[2 K r_0\le d(z,x_0)\le K(  d(z,y)+ d(y,x_0)) < K(2 K r_1 + r_0).\]
This implies that $B\subseteq B(x_1,5K^3r_1)$. To see this, notice that for  $x\in B$ we have
\[d(x,x_1)\le  K^2(d(x,x_0)+ d(x_0, y) + d(y, x_1))<K^2(2r_0+r_1),\]
so  that $x\in B(x_1,5K^3r_1)$. Therefore
\begin{align*}
\frac{\nu_2(B_1)}{\mu(B_1)}
&\le \frac{\mu(B(x_1,5K^3r_1))}{\mu(B_1)}
 M_\mu \nu(x).
\end{align*}
Since $B_1$ is an arbitrary $d$-ball containing $y$,
\[
M_\mu\nu_2(y)\le\frac{\mu(B(x_1,5K^3r_1))}{\mu(B_1)} M_\mu \nu(x).
\]
Integrating with respect to $y$
we obtain
\begin{equation*}\label{f_2}
 \frac{1}{\mu(B)}\int_{B} (M_\mu\nu_2)^\gamma\, d\mu \le C M_\mu\nu(x)^\gamma
\end{equation*}
and the theorem is proved.
\end{proof}

The first elementary use of
Theorem~\ref{teo:maximal en A1} to produce  weights in the $A_p(X,d,\mu)$ class is obtained taking  $\nu=\delta_0$ the Dirac delta measure in $x_0$. In this case we have
\[
M_\mu\delta_0(x)= \sup_{r>0} \frac{\delta_0(B(x,r))}{\mu(B(x,r))}= \frac{1}{\mu(\overline{B}(x,d(x,x_0))},
\]
where $\overline{B}(x,s)=\{y\in X:d(x,y)\leq s\}$.
Since for the doubling condition we have $\mu(B(x,s))\leq \mu(\overline{B}(x,s))\leq A \mu(B(x,s))$, applying Theorem ~\ref{teo:maximal en A1} we obtain that $\mu(B(x,d(x,x_0)))^{-\gamma}\in A_1(X, d,\mu)$ for $0\le \gamma <1$. Consequently
 $\mu(B(x,d(x,x_0)))^{\beta}\in A_p(X, d,\mu)$ for $-1<\beta < p-1$.

\medskip
As an immediate consequence of Proposition~\ref{propo: s-sets miden cero} and Theorem~\ref{teo:maximal en A1}, we have the following result.

\begin{corollary}\label{coro: maximal finita}
Let  $(X,d)$ be  $\alpha$-Ahlfors with measure $\mu$ and $F$ closed in $X$. If $(F,d)$ is  $s$-Ahlfors with measure $\nu$, where $0\leq s<\alpha$, then
  $(M_\mu\nu)^\gamma\in A_1(X,d,\mu)$ for every $0\le \gamma <1$. The same is true when
 $F=\bigcup_{i=1}^H F_i$ and $\nu=\sum_{i=1}^H \nu_i$, where
$\{F_1,\dots ,F_H\}$ is a family of  pairwise disjoint closed and bounded subsets of $X$ such that $(F_i,d)$ is $s_i$-Ahlfors with measure $\nu_i$, where $0\leq s_i<\alpha$ for $i=1,2,\dots ,H$. Moreover the weight $\max\{M_\mu^\gamma\nu(x),C\} \in A_1(X,d,\mu)$ for every $C>0$.
\end{corollary}

In the following result we  explore the behavior of $M_\mu\nu$ to obtain an explicit family of Muckenhoupt weights.
 Here we use the notation $f(x)\simeq g(x)$ to indicate that there exist positive and finite constants $k_1$ and $k_2$ such that $k_1g(x)\leq f(x)\leq k_2g(x)$ for every~$x$.

\begin{theorem}\label{teo: X normal, F s-set}
Let   $(X,d)$ be  $\alpha$-Ahlfors with measure $\mu$. Let $\{F_1,\dots ,F_H\}$ be a family of  pairwise disjoint  closed and bounded  subsets of $X$ such that $(F_i,d)$ is locally $s_i$-Ahlfors with measure $\nu_i$, where $0\leq s_i<\alpha$  for $i=1,2,\dots ,H$. Set $F=\bigcup_{i=1}^H F_i$ and $\nu=\sum_{i=1}^H \nu_i$. Then  there exist  open sets $U_1,\dots, U_H$ pairwise disjoint with $U_i$ containing $F_i$ such that
\begin{enumerate}
\item\label{item: cerca} if  $d(x,F)< 2K\diam(F)$ then we have
\begin{enumerate}
   \item\label{item: en Ui}  $M_\mu\nu(x)\simeq d(x,F_i)^{s_i-\alpha}$, for every $x\in U_i\setminus F_i$, for every $i=1,2,\dots ,H$;
   \item\label{item: fuera de Ui} if $x\notin \bigcup_{i=1}^H U_i$ then $M_\mu\nu(x)\simeq 1$;
 \end{enumerate}
\item\label{item: lejos} if $d(x,F)\geq 2K\diam(F)$ then $M_\mu\nu(x)\simeq d(x,F)^{-\alpha}\simeq d(x,x_0)^{-\alpha}$ for every $x_0\in F$.
\end{enumerate}
\end{theorem}

\begin{proof} Let us start by proving (\ref{item: lejos}).
 Let us fix $x$ such that $d(x,F)\geq 2K\diam(F)$. Since $F$ is bounded, we only need to prove   the first equi\-va\-lence in (\ref{item: lejos}).
 Since for $r\leq d(x,F)$ we have $B(x,r)\cap F=\emptyset$, in order to estimate $M_\mu\nu(x)$ we only have to consider balls centered at $x$ with radius $r>d(x,F)$.
 Hence
 \[\frac{\nu(B(x,r))}{\mu(B(x,r))}\leq \frac{c\nu(F)}{d(x,F)^\alpha},\]
 where $c$ denotes a constant for the Ahlfors condition of $\mu$.

 On the other hand, taking $r=3Kd(x,F)$ we obtain
 \[M_\mu\nu(x)\geq \frac{\nu(B(x,3Kd(x,F)))}{c(3Kd(x,F))^\alpha}. \]
Now fix $y\in F$  such that $d(x,y_1)<\frac32 d(x,F)$. Then \[B(y,2K\diam(F))\subseteq B(y,d(x,F))\subseteq B(x,3Kd(x,F)),\]
 and therefore
 \[M_\mu\nu(x)\geq \frac{\nu(B(y,2K\diam(F)))}{c(3Kd(x,F))^\alpha}= \frac{\nu(F)}{c(3Kd(x,F))^\alpha}. \]


To prove (\ref{item: cerca}), let $c_i$ be the constant for the Ahlfors condition of $\nu_i$. In other words, $c_i$ is a constant satisfying
\[c_i^{-1}r^{s_i}\leq \nu_i(B(y,r))\leq c_i r^{s_i},\]
for every $0< r\leq r_i$ and every $y\in F_i$, for some $r_i>0$.
Let us define $\Delta=\min\{d(F_i,F_j): i\neq j\}$ and
\[U_i=\left\{x\in X: d(x,F_i)<\kappa_i\right\},\]
where $\kappa_i\leq \min\{2r_i,\Delta/2\}$ is a constant that we shall define later.
Notice that $U_i\cap U_j=\emptyset$ for $i\neq j$.

In order to show  (\ref{item: en Ui}), let us fix $i\in\{1,\dots,H\}$ and $x\in U_i\setminus F_i$.
 To obtain the lower bound, let $y\in F_i$ such that $d(x,y)<\frac{3}{2}d(x,F_i)$. Then
 \begin{align*}
M_\mu\nu(x)&\geq \frac{\nu\left(B\left(x,2Kd(x,F_i)\right)\right)}{\mu\left(B\left(x,2Kd(x,F_i)\right)\right)}
\\
&\geq
 \frac{\nu_i\left(B\left(x,2Kd(x,F_i)\right)\right)}{\mu\left(B\left(x,2Kd(x,F_i)\right)\right)}\\
&\geq  \frac{\nu_i\left(B\left(y,\frac{1}{2}d(x,F_i)\right)\right)}{c (2Kd(x,F_i))^\alpha}
\\
&\geq  C_i d(x,F_i)^{s_i-\alpha},
  \end{align*}
where $ C_i=\frac{1}{c c_i 2^{s_i}(2K)^\alpha}$.

To obtain the upper bound, let us keep consider the case $x\in U_i\setminus F_i$. Hence if $r\leq d(x,F_i)$ we have that $r<\Delta/2$, so that $B(x,r)\cap F=\emptyset$ and $\nu(B(x,r))=0$. Assume then
 $r>d(x,F_i)$ and
 fix $y\in F_i$ such that $d(x,y)<r$. Then
\[\frac{\nu(B(x,r))}{\mu(B(x,r))}\leq \frac{\nu(B(y,2Kr))}{\mu(B(x,r))}\leq \frac{c \nu(B(y,2Kr))}{r^\alpha}.\]
We shall consider two cases taking into account the size of $r$.

\noindent\textbf{Case 1: } $d(x,F_i)<r<\frac{\Delta}{2K}$. In this case we have $2Kr<\Delta$, so that $B(y,2Kr)\cap F_j=\emptyset$ for every $j\neq i$. Then
\[\frac{\nu(B(y,2Kr))}{r^\alpha}=\frac{\nu_i(B(y,2Kr))}{r^\alpha}\leq c_i (2Kr)^{s_i} r^{-\alpha}<c_i (2K)^{s_i} d(x,F_i)^{s_i-\alpha}.\]

\noindent\textbf{Case 2: } $r\geq\frac{\Delta}{2K}$. In this case
 \[\frac{\nu(B(x,r))}{\mu(B(x,r))}\leq  \frac{c \nu(F)}{r^\alpha}\leq c\left(\frac{2K}{\Delta}\right)^\alpha\nu(F).\]

To obtain the result we only need to take $x$ such that $c_i (2K)^{s_i} d(x,F_i)^{s_i-\alpha}\geq c\left(\frac{2K}{\Delta}\right)^\alpha\nu(F)$. In other words, it is enough to take
\[\kappa_i=\min\left\{2r_i,\frac{\Delta}{2}, \left(\frac{c2^\alpha K^\alpha\nu(F)}{c_i\Delta^\alpha2^{s_i} K^{s_i}}\right)^{\frac{1}{s_i-\alpha}}\right\}\]
in the definition of $U_i$.

Finally, to prove (\ref{item: fuera de Ui}) take $x\notin\bigcup_{i=1}^H U_i$ such that $d(x,F)< 2K\diam(F)$. Then
$d(x,F_i)\geq\kappa_i$ for each $i$, so that $d(x,F)\geq \min \kappa_i=:\kappa$.  As before, in order to estimate $M_\mu \nu(x)$ we only need to consider $r>d(x,F)\geq \kappa$. In this case
\[\frac{\nu(B(x,r))}{\mu(B(x,r))}\leq \frac{c \nu(F)}{\kappa^\alpha}.\]
On the other hand, since $d(x,F)< 2K\diam(F)$ we have  $F\subseteq B(x,3K^2\diam(F))$ and therefore
\[M_\mu\nu(x)\geq \frac{\nu(B(x,3K^2\diam(F)))}{\mu(B(x,3K^2\diam(F)))}\geq \frac{\nu(F)}{c(3K^2\diam(F))^\alpha}.\]
\end{proof}


As a consequence of Corollary~\ref{coro: maximal finita} and Theorem~\ref{teo: X normal, F s-set}, we have the following result.

\begin{theorem}\label{teo: peso en A1 para alpha-normales multiple}
Let   $(X,d)$ be  $\alpha$-Ahlfors with measure $\mu$. Let $\{F_1,\dots ,F_H\}$ be a family of  pairwise disjoint  closed and bounded  subsets of $X$ such that $(F_i,d)$ is a locally $s_i$-Ahlfors space with measure $\nu_i$, where $0\leq s_i<\alpha$  for $i=1,2,\dots ,H$.  Then
\begin{enumerate}
 \item \label{item1}
\[
w(x)=\left\{
\begin{array}{ll}
d(x,F_i)^{(s_i-\alpha)\gamma},& \textrm{ for } x\in U_i ;\\
1,& \textrm{ for } x\in \left(\bigcup_{i=1}^H U_i\right)^c
\end{array}
\right.\]
belong to $ A_1(X, d, \mu)$ for every $0\leq \gamma<1$, where $U_i$ is the open set containing $F_i$   given by Theorem \ref{teo: X normal, F s-set};

\item \label{item2}
\[
v(x)=\left\{
\begin{array}{ll}
d(x,F_i)^{\beta_i},& \textrm{ for } x\in U_i ;\\
1,& \textrm{ for } x\in \left(\bigcup_{i=1}^H U_i\right)^c
\end{array}
\right.\]
belongs to $A_p(X, d, \mu)$ for every $-(\alpha -s_i)< \beta_i<(\alpha -s_i)(p-1).$
\end{enumerate}
\end{theorem}

\begin{proof}
Since the equivalence $w_1\simeq w_2$ preserves $A_1(X, d,\mu)$, in order to prove \eqref{item1}, from Corollary \ref{coro: maximal finita} we only need to check that $w(x)\simeq \max\{M^\gamma_\mu \nu (x), C\}=: u(x)$ for some positive constant $C$.

As  we saw in item \eqref{item: en Ui} in Theorem~\ref{teo: X normal, F s-set}, for each $i=1,...,H$, there exists a constant  $ \tilde{C}$ such that if $x\in U_i\setminus F_i$ then $M_\mu\nu(x)\geq  \tilde C d(x,F_i)^{s_i-\alpha}$ and, since  $d(x,F_i)\leq 2r_i$, taking $s^*=\max\{s_1,\dots,s_H\}$ and $r^*= \max\{r_1,\dots,r_H\}$, we have that $M^\gamma_\mu\nu(x)\geq \tilde C^\gamma (2r^*)^{\gamma(s^*-\alpha)}$.

%
 Then, if we will consider $C:=\tilde C^\gamma (2r^*)^{\gamma(s^*-\alpha)}$ in the definition of $u$ we obtain that  for $x\in \bigcup_{i=1}^H U_i\setminus F_i$ we have
%
%
%
\[u(x)=M_\mu^\gamma\nu(x)\simeq d(x,F_i)^{(s_i-\alpha)\gamma}=w(x).\]

For the case of $x\notin \bigcup_{i=1}^H U_i$,
$w(x)=1$, and to see that $u(x)\simeq 1$ we shall consider the following possibilities for  $x$:
\begin{itemize}
\item $d(x,F)<2K\diam(F)$. By \eqref{item: lejos} of Theorem \ref{teo: X normal, F s-set}, $M_\mu^\gamma\nu(x)\simeq 1$ and  we have that $u(x)\simeq w(x)$, not depending on the value of $u(x)$.

%
\item $2K\diam(F)<d(x,F)\leq\Lambda$, with \[\Lambda=\max\left\{2K\diam(F), \left(\frac{c\nu(F)}{C^{1/\gamma}}\right)^{1/\alpha}\right\}\] and  $c$ a constant for the Ahlfors condition of $\mu$. Item~(\ref{item: lejos}) of Theorem~\ref{teo: X normal, F s-set} together with the bounds for  $d(x,F)$ implies that $M_\mu^\gamma\nu(x)\simeq d(x,F)^{-\alpha\gamma}\simeq 1$. Then, as in the previous case, $u(x)\simeq 1$.
\item $d(x,F)>\Lambda$. In the proof of (\ref{item: lejos}) in  Theorem~\ref{teo: X normal, F s-set} we can see that $M_\mu\nu(x)\leq \frac{c\nu(F)}{d(x,F)^\alpha}$. Then
    \[M_\mu^\gamma\nu(x)\leq \left(\frac{c\nu(F)}{\Lambda^\alpha}\right)^\gamma\leq C.\]
Hence $u(x)=C\simeq 1$.
\end{itemize}

Item \eqref{item2} is an immediate consequence of item \eqref{item1} and the definition of Muckenhoupt class.
\end{proof}

In the particular case of an $s$-Ahlfors space $F$, we can  take out the restriction of boundedness and we obtain the following result.

\begin{theorem}\label{coro: peso en A1 para alpha-normales}
Let  $(X,d)$ be  $\alpha$-Ahlfors with measure $\mu$ and $F$ closed in $X$. If $(F,d)$ is  $s$-Ahlfors with measure $\nu$, where $0\leq s<\alpha$, then
$d(x,F)^{\gamma(s-\alpha)}\in A_1(X, d, \mu)$ for every $0\leq \gamma<1$. Consequently
  $d(x,F)^{\beta}\in A_p(X, d,\mu)$ for $-(\alpha-s)<\beta < (\alpha-s)(p-1)$.
\end{theorem}

\begin{proof}
 The results follows if we prove that $M_\mu\nu(x)\simeq d(x,F)^{s-\alpha}$ for $\mu$-almost every $x\in X$.

Take $x\notin F$.
The upper bound was proved in Proposition \ref{propo: s-sets miden cero}. To obtain the lower bound, let $y\in F$ be such that $d(x,y)<\frac32 d(x,F)$. Then
\[
M_\mu\nu(x)\geq
 \frac{\nu\left(B\left(x,2Kd(x,F)\right)\right)}{\mu\left(B\left(x,2Kd(x,F)\right)\right)}
\geq  \frac{ \nu\left(B\left(y,\frac{1}{2}d(x,F)\right)\right)}{c(2Kd(x,F))^\alpha}
\geq C  d(x,F)^{s-\alpha},
\]
where $C=\frac{c}{\tilde{c} (2K)^\alpha 2^s}$, with $c$ and $\tilde c$  constants for the Ahlfors condition of $\mu$ and $\nu$ respectively.

For $x\in F$, as we proved in Proposition \ref{propo: s-sets miden cero},  $M_\mu\nu(x)=\infty$ but $\mu(F)=0$.
\end{proof}

\section{Applications}\label{section: applications}
This section contain three topics. The first one deal with weighted Sobolev regularity for solutions of elliptic partial differential equations of large order.

Second we give a criteria to decide wether or not a given set is locally $s$-Ahlfors in terms of the $s$-dimensional Hausdorff measure.

Third, we provide  some non trivial $A_p$ weights on the Sierpinski gasket built from Theorem~\ref{coro: peso en A1 para alpha-normales}.\\

\noindent\textbf{Weighted Sobolev regularity. }
 Let $(X, d)$  be  $\alpha$-Ahlfors with measure $\mu$ and $\Omega$ subset of $X$ with $\diam (\Omega) <\infty$.

Given an operator $R: \Omega\times \Omega\longrightarrow \RR$ and $f:\Omega\longrightarrow \RR$ we define
\[
 u(x):=\int_{\Omega} R(x,y)\, f(y)\, d\mu(y).
\]

If there exists a positive constant $C$ such that
\begin{equation}\label{cotaR}
 |R(x,y)|\le C d(x,y)^{-\alpha+1},
\end{equation}
we can see that $u$ satisfies
\[
| u(x)|\le C \,   M_\mu f(x),
\]
where $M_\mu f(x)$ is the Hardy-Littlewood maximal operator in $(X,d,\mu)$.
In fact, if we denote $\delta$ the diameter of $\Omega$ and using  (\ref{cotaR})
 we have
\begin{eqnarray*}
 |u(x)|&\le& C \int_{d(x,y)\le\delta} \frac{|f(y)|}{d(x,y)^{\alpha-1}}\, d\mu(y)
\\
&=& C\, \sum_{k=0}^{\infty}\int_{\{{2^{-(k+1)}\delta\le d(x,y)\le 2^{-k}\delta}\}} \frac{|f(y)|}{d(x,y)^{\alpha-1}}\, d\mu(y)
\\
&\le & C\, 2^{\alpha-1} \sum_{k=0}^{\infty} \frac{2^{-k} }{\mu(B(x,2^{-k}\delta))}\int_{\{d(x,y)\le 2^{-k}\delta\}} |f(y)| d\mu(y)
\\
&\le& C \,   M_\mu f(x).
\end{eqnarray*}

Since $M_\mu f(x)$  is a bounded operator in
$L^p(\Omega,wd\mu)$  if $w\in A_p(X,d,\mu)$,
follows immediately that
\begin{equation}\label{estimate_for_the_solution}
\|u\|_{L^p(\Omega,wd\mu)}
\le C\, \|f\|_{L^p(\Omega,wd\mu)}
\end{equation}
i.e. there exists a positive constant $C$ such that
\[
 \left(\int_\Omega |u|^p\, w\, d\mu\right)^{1/p}\le C\, \left( \int_\Omega |f|^p\, w\, d\mu\right)^{1/p},
\]
provided that $w\in A_p(X,d,\mu)$.

As in general we are interested in sources $f$ as singular as possible the above inequality allows unbounded growth of $f$ close to some subset of $\Omega$ if the weight $w$ vanishes there. This is precisely the case for weights produced by
 Theorem~\ref{teo: peso en A1 para alpha-normales multiple}.

\medskip

In the particular case of the $n$-dimensional Euclidean space $(\RR^n, |\cdot|, \lambda)$, taking $R$ as the Green function associated
with the  polyharmonic Dirichlet problem in $\Omega$
\begin{eqnarray}\label{DP}\left\{\begin{array}{ccc}
(-\Delta)^m u=f &\mbox{ in }\Omega\\
\left(\frac{\partial}{\partial \nu}\right)^{j}u=0 &\mbox{ on
}\partial\Omega & 0\leq j\leq m-1,
\end{array}\right.
\end{eqnarray}
where $\frac{\partial}{\partial \nu}$ is the normal derivative, it is known that
the solution of \eqref{DP} is given by
\begin{equation}
\label{u1}
u(x)=\int_\Omega R(x,y)\, f(y)\, dy.
\end{equation}
Then, to prove some weighted Sobolev a priori estimates for the solution of this problem we need estimates for the Green function $R(x,y)$ and its derivatives.

For $\Omega$ a regular domain in $\RR^n$,  some of this estimates was given in \cite{DuranSanmartinoToschi2}, where   the authors proved that
\begin{equation}\label{estimate_for_the_solutionDST}
\|u\|_{W^{2m,p}(\Omega,wdx)}
\le C\, \|f\|_{L^p(\Omega,wdx)}
\end{equation}
for $w\in A_p$.

Let us remember that for $\eta$ a multi-index,
$\eta=(\eta_1, \eta_2, \dots,\eta_n)\in {\rm Z}\!\!{\rm
Z}^n_+$ we denote as usual $|\eta|=\sum_{j=1}^n \eta_j,$ $ D^{\eta}=\partial_{x_1}^{\eta_1}...\partial_{x_n}^{\eta_n}$. The Sobolev spaces are defined by
\[
W^{k,p}_w(\Omega)=\{v\in L^p(\Omega,wdx)\,:\, D^\eta v\in L^p(\Omega,wdx) \quad
\forall \, |\eta|\le k\}
\]
and the norm of $v\in W^{k,p}_w(\Omega)$ is given by
\[\|v\|_{W^{k,p}_w(\Omega)}=\sum_{|\eta|\leq k}\|D^\eta v\|_{L^p(\Omega,wdx)}\, .\]

\medskip

For $\Omega$ a general bounded open set in $\RR^n$
there exists a constant $C$ such that the
Green function satisfies the estimate
\begin{equation}
\label{cotaDG} |D^\eta R(x,y)|\le C |x-y|^{2m-n-|\eta|},
\end{equation}
 for $n\in [3,2m+1]\cap \NN$ odd and $0\le |\eta|\le m-\frac{n}{2}+\frac12$ and for
$n\in [2,2m]\cap \NN$ even and  $0\le |\eta|\le m-\frac{n}{2}$ (see \cite{MM}).

Then, since $|\eta|\le 2m-1$, it follows that
\begin{equation}\label{estimate_for_the_solution}
\|u\|_{W^{k,p}_w(\Omega)}
\le C\, \|f\|_{L^p(\Omega,wdx)}
\end{equation}
for values of $n$, $m$ and $k=|\eta|$ given above, provided that $w\in A_p(\mathbb R^n)$. \\

\noindent\textbf{Hausdorff measure based criteria for the local $\boldsymbol s$-Ahlfors condition.}
In the hypothesis of the results obtained in the previous section we require that the spaces are locally $s$-Ahlfors.
In order to check that a given  $(F,d)$ is  locally $s$-Ahlfors, we should be able to find a Borel measure  $\nu$ supported on $F$ and a real number $r_0>0$ such that (\ref{eq: normal}) holds for every $x\in F$ and every $0<r<r_0$. This does not seem to be an easy task. However, if $(F,d)$ is  locally $s$-Ahlfors, then there exists essentially only one Borel  measure $\nu$ satisfying the condition required in the definition. This fact is known in the Euclidean setting (see for instance \cite{Triebel}), and for the sake of completeness let us extend it to general metric measure spaces in Lemma~\ref{lemma: about s-sets}. First we shall need  recall some definitions. The basic aspects of Hausdorff measure and dimension can be found in \cite{Falconer}.

\smallskip

For $\rho>0$, we say that a sequence $\{E_i\}$ of subsets of $X$ is a \emph{$\boldsymbol\rho$\textbf{-cover}} of a set $E$ if
$E\subseteq \bigcup E_i$ and $\diam(E_i)\leq
\rho$ for every $i$. Let $E\subseteq X$ and $s\geq 0$ fixed.
We define
\[\mathscr{H}^s_{\rho}(E)=\inf\left\{\sum_{i=1}^{\infty}\diam^s(E_i): \{E_i\}
\textrm{ is a $\rho$-cover of } E\right\}.\]
Clearly  $\mathscr{H}^s_{\rho}(E)$ increases when  $\rho$ decreces, so that the limit  when  $\rho$ tends to $0$ exists
 (although it may be infinite). Then we define
\[\mathscr{H}^s(E)=\lim_{\rho \to 0}\mathscr{H}^s_{\rho}(E)=\sup_{\rho>0}
\mathscr{H}^s_{\rho}(E).\]
We shall refer to  $\mathscr{H}^s(E)$ as the
 \emph{$\boldsymbol s$\textbf{-dimensional Hausdorff measure}} of~$E$.

If in the above definition we replace the arbitrary $\rho$-cover by coverings by $d$-balls centering in the set $E$ and with diameter less than $\rho$, we obtain the \emph{$s$-dimensional spherical Hausdorff measure} of $E$, which will denote $\mathscr S^s$. 

\begin{lemma}\label{lemma: about s-sets} Let $(X,d)$ be a quasi metric space.
\begin{enumerate}
\item\label{item: equivalentes}
For every $s>0$ and every $E\subseteq X$ we have that \[\mathscr H^s(E)\leq \mathscr S^s(E)\leq K^s2^s\mathscr H^s(E),\] where $K$ denotes de triangle constant for $d$. In other words, the measures $\mathscr H^s$ and $\mathscr S^s$ are equivalents.
\item\label{item: radios por diametros} If $F$ is locally $s$-Ahlfors with measure $\nu$, then there exists a constant $\tilde c$ such that
    \[{\tilde{c}}^{-1} (\diam(B(x,r)))^s\le \nu(B(x,r)) \le \tilde{c} (\diam(B(x,r)))^s\] for every $x\in F$ and every $0<r<r_0$.
\item\label{item:s-set con la de hausdorff} If $(X,d)$ has finite metric dimension and $(F,d)$ is locally $s$-Ahlfors with measure $\nu$, then $(F,d)$ is locally $s$-Ahlfors with the restriction of  $\mathscr H^s$ to~$F$.
    \end{enumerate}
\end{lemma}

\begin{proof}
The first inequality in (\ref{item: equivalentes}) is straightforward. For the second one,
 let us fix $\rho>0$ and let $\{E_i\}$ a $\rho$-cover of $E$. We can assume that $E_i\cap E\neq \emptyset$ for every $i$.
 Let $\varepsilon>0$ given. For each $i$, let us fix $x_i\in E_i\cap E$ and set $B_i=B(x_i,\diam(E_i)+\varepsilon)$. Then $E_i\subseteq B_i$, so that $\{B_i\}$ is a covering of $E$ by $d$-balls with diameter less than  $2K(\rho+\varepsilon)$. Then
 \[\mathscr S^s_{2K(\rho+\varepsilon)}(E)\leq \sum_{i} \diam^s(B_i)<2^sK^s\sum_{i} \diam^s(E_i) + 2^sK^s \varepsilon^s.\]
 Hence $\mathscr S^s_{2K(\rho+\varepsilon)}(E)\leq 2^sK^s \mathscr H^s_\rho(E)+ 2^sK^s \varepsilon^s$. Taking $\varepsilon=\rho$  and by making $\rho\to 0$ we obtain the result.

To prove (\ref{item: radios por diametros}), let $c$ such that
$c^{-1} r^s\le \nu(B(x,r)) \le c r^s$  for every $x\in F$ and every  $0<r<r_0$.
Fix $x\in F$ and $0<r<r_0$ and set
 $B=B(x,r)$. Since
 $\diam(B)\le 2K r$, the first inequality follows immediately with $\tilde c=c2^{s}K^{s}$. To obtain the second one, let us consider
  two cases
 according to the size of $\diam(B)$.
If $\diam(B)<r_0$  there exists $0<\varepsilon<\diam(B)$ such that $\diam(B) +\varepsilon<r_0$.
Also for each $x\in F$  we have $B(x,r)\cap F\subset B(x,\diam(B) +\varepsilon)\cap F$.
Then
\begin{equation}\label{eq: radios por diametros}
\nu(B(x,r))\le \nu(B(x,\diam(B) +\varepsilon))\le c 2^s(\diam(B) )^s.
\end{equation}
Otherwise, if  $\diam(B)\ge r_0$, since $r<r_0$ we have $r<\diam(B)$. So that $\tilde c=c2^sK^s  $ works.

Finally, in order to prove (\ref{item:s-set con la de hausdorff}),
let us fix $x\in F$ and $0<r<r_0$, and set $F(x,r)$ to denote the set $F\cap B(x,r)$. By hypothesis there exists a constant   $c$  such that
\[c^{-1} r^s\leq \nu(F(x,r))\leq c r^s.\]
Let $\{B_j\}$ be a covering of $F(x,r)$ by $d$-balls centering in $F(x,r)$, and let $r_j$ the radio of $B_j$. Since we can assume $0<r_j<r_0$ for every $j$, from (\ref{eq: radios por diametros}) we have that
\[ \nu(B_j\cap F)\leq c  2^s (\diam(B_j))^s.\]
Then
\[c^{-1}r^s\leq \nu(F(x,r))\leq\sum_j \nu(B_j\cap F)\leq  c 2^s \mathscr S^s(F(x,r))\leq c 4^{s} K^s \mathscr H^s(F(x,r)).\]
On the other hand, we claim that there exists a constant $\Lambda$, which does not depend on $x$ or $r$, such that  $\mathscr H^s(F(x,r))\leq \Lambda r^s$.
 To prove it we shall consider two possibilities. First let us assume that $0<r<r_0/K$, and let $0<\rho<r_0/K-r$. Set $U$ a   finite $\rho$-net in  $F(x,r)$, let us say $U=\{x_1,x_2,\dots,x_{I_\rho}\}$.  Then $\{B(x_i,\rho): i=1,\dots,I_\rho\}$ is a $2K\rho$-cover of $F(x,r)$ and
 each $y\in F(x,r)$ belongs to at most  $N$ of such balls, where $N$ is the constant from the finite metric dimension of $X$, which  does not depend on $\rho$, $r$ or $x$. In fact, for a fixed  $y\in F(x,r)$, we have that $y\in B(x_i,\rho)$ if and only if $x_i\in B(y,\rho)$, so that the number of balls $B(x_i,\rho)$ to which $y$ belongs is equal to the cardinal of $U\cap B(y,\rho)$. Then we have
 \begin{align*}
   \mathscr H_{2K\rho}^s(F(x,r))&\leq \sum_{i=1}^{I_\rho} \rho^s\\&\leq c \sum_{i=1}^{I_\rho} \nu\left(B(x_i,\rho)\right)\\
   &\leq c N\nu\left(\bigcup_{i=1}^{I_\rho} B(x_i,\rho)\right)\\ &\leq c N\nu\left(B(x,K(r+\rho))\right)\\
    &\leq c^2 N K^s(r+\rho)^s.
 \end{align*}
%
Taking $\rho\to 0$ we obtain the desired result for this case. Then only remains to consider the case $\frac{r_0}{K}\leq r<r_0$. In this case, being  $B(x,r)$ a bounded set, there exists a finite $r_0(2K)^{-1}$-net in $B(x,r)$, let us say $U=\{x_1,\dots, x_{I_{r_0}}\}$.
Then $F(x,r)\subseteq \bigcup_{i=1}^{I_{r_0}} B\left(x_i,\frac{r_0}{2K}\right)$. Applying the previous case we obtain
 \[\mathscr H^s(F(x,r))\leq \sum_{i=1}^{I_{r_0}} \mathscr H^s\left(
 B\left(x_i,\frac{r_0}{2K}\right)\cap F\right)\leq I_{r_0} \Lambda \left(\frac{r_0}{2K}\right)^s\leq I_{r_0}\Lambda 2^{-s} r^s.  \]
Moreover, we have that $I_{r_0}\leq N^{1+\log_2K}$, since  every $s$-disperse subset of $X$ has at most $N^m$ points in each
ball of radius $2^m s$, for all $m\in\mathbb N$ and every $s>0$ (see
\cite{CoifWeiss} and \cite{Asso}). 
\end{proof}

\textbf{Muckenhoupt weights on the Sierpinski gasket.}
We shall apply  Theorem~\ref{coro: peso en A1 para alpha-normales} to a classical fractal set.
Set $T$  any equilateral triangle in $\mathbb R^2$.
Let $S$ be the   Sierpinski's gasket constructed in $T$ equipped with the usual distance $d$ inherited from $\mathbb R^2$ and with the $\alpha$-dimensional Hausdorff measure $\mathscr H^\alpha$, where $\alpha=\frac{\log 3}{\log 2}$. It is well known that $(S,d)$ is $\alpha$-Ahlfors with measure $\mathscr H^\alpha$ (see \cite{Mosco}). Set $F$  the  boundary of the triangle $T$. Then $F$ is  $1$-Ahlfors with measure $\nu$,  taking $\nu$ as the length.  So that Theorem~\ref{coro: peso en A1 para alpha-normales} says that
\[d(x,F)^{(1-\alpha)\gamma}\in A_1(S,d,\mathscr H^\alpha),\]
for every $0\leq\gamma < 1$.


\def\cprime{$'$}


\begin{thebibliography}{GCRdF85}

\bibitem[Ass79]{Asso}
Patrice Assouad.
\newblock \'{E}tude d'une dimension m\'etrique li\'ee \`a la possibilit\'e de
  plongements dans {${\bf R}\sp{n}$}.
\newblock {\em C. R. Acad. Sci. Paris S\'er. A-B}, 288(15):A731--A734, 1979.

\bibitem[CW71]{CoifWeiss}
Ronald~R. Coifman and Guido Weiss.
\newblock {\em Analyse harmonique non-commutative sur certains espaces
  homog\`enes}.
\newblock Springer-Verlag, Berlin, 1971.
\newblock \'Etude de certaines int\'egrales singuli\`eres, Lecture Notes in
  Mathematics, Vol. 242.

\bibitem[dG81]{DeGuz}
Miguel de~Guzm\'an.
\newblock {\em Real variable methods in {F}ourier analysis}, volume~46 of {\em
  North-Holland Mathematics Studies}.
\newblock North-Holland Publishing Co., Amsterdam, 1981.

\bibitem[DLG10]{DuranLopez}
Ricardo~G. Dur{\'a}n and Fernando L{\'o}pez~Garc{\'{\i}}a.
\newblock Solutions of the divergence and analysis of the {S}tokes equations in
  planar {H}\"older-{$\alpha$} domains.
\newblock {\em Math. Models Methods Appl. Sci.}, 20(1):95--120, 2010.

\bibitem[DS04]{D-S}
A.~Dall'Acqua and G.~Sweers.
\newblock Estimates for {G}reen function and {P}oisson kernels of higher-order
  {D}irichlet boundary value problems.
\newblock {\em J. Differential Equations}, 205(2):466--487, 2004.

\bibitem[DST08]{DuranSanmartinoToschi}
R.~G. Dur{\'a}n, M.~Sanmartino, and M.~Toschi.
\newblock Weighted a priori estimates for the {P}oisson equation.
\newblock {\em Indiana Univ. Math. J.}, 57(7):3463--3478, 2008.

\bibitem[DST10]{DuranSanmartinoToschi2}
Ricardo~G. Dur{\'a}n, Marcela Sanmartino, and Marisa Toschi.
\newblock Weighted a priori estimates for solution of {$(-\Delta)\sp mu=f$}
  with homogeneous {D}irichlet conditions.
\newblock {\em Anal. Theory Appl.}, 26(4):339--349, 2010.

\bibitem[Duo01]{Javi}
Javier Duoandikoetxea.
\newblock {\em Fourier analysis}, volume~29 of {\em Graduate Studies in
  Mathematics}.
\newblock American Mathematical Society, Providence, RI, 2001.
\newblock Translated and revised from the 1995 Spanish original by David
  Cruz-Uribe.

\bibitem[Fal86]{Falconer}
K.~J. Falconer.
\newblock {\em The geometry of fractal sets}, volume~85 of {\em Cambridge
  Tracts in Mathematics}.
\newblock Cambridge University Press, Cambridge, 1986.

\bibitem[GCRdF85]{garcia-rubio}
Jos{\'e} Garc{\'{\i}}a-Cuerva and Jos{\'e}~L. Rubio~de Francia.
\newblock {\em Weighted norm inequalities and related topics}, volume 116 of
  {\em North-Holland Mathematics Studies}.
\newblock North-Holland Publishing Co., Amsterdam, 1985.
\newblock , Notas de Matem\'atica [Mathematical Notes], 104.

\bibitem[Gro07]{gromov}
Misha Gromov.
\newblock {\em Metric structures for {R}iemannian and non-{R}iemannian spaces}.
\newblock Modern Birkh\"auser Classics. Birkh\"auser Boston Inc., Boston, MA,
  english edition, 2007.
\newblock Based on the 1981 French original, With appendices by M. Katz, P.
  Pansu and S. Semmes, Translated from the French by Sean Michael Bates.

\bibitem[Jon80]{jones}
Peter~W. Jones.
\newblock Factorization of {$A\sb{p}$} weights.
\newblock {\em Ann. of Math. (2)}, 111(3):511--530, 1980.

\bibitem[MM09]{MM}
Svitlana Mayboroda and Vladimir Maz'ya.
\newblock Pointwise estimates for the polyharmonic {G}reen function in general
  domains.
\newblock In {\em Analysis, partial differential equations and applications},
  volume 193 of {\em Oper. Theory Adv. Appl.}, pages 143--158. Birkh\"auser
  Verlag, Basel, 2009.

\bibitem[Mos97]{Mosco}
Umberto Mosco.
\newblock Variational fractals.
\newblock {\em Ann. Scuola Norm. Sup. Pisa Cl. Sci. (4)}, 25(3-4):683--712
  (1998), 1997.
\newblock Dedicated to Ennio De Giorgi.

\bibitem[MS79]{M-S}
Roberto~A. Mac{\'{\i}}as and Carlos Segovia.
\newblock Lipschitz functions on spaces of homogeneous type.
\newblock {\em Adv. in Math.}, 33(3):257--270, 1979.

\bibitem[Muc72]{Muck}
Benjamin Muckenhoupt.
\newblock Weighted norm inequalities for the {H}ardy maximal function.
\newblock {\em Trans. Amer. Math. Soc.}, 165:207--226, 1972.

\bibitem[Sj{\"o}97]{Sjoding}
Tord Sj{\"o}din.
\newblock On {$s$}-sets and mutual absolute continuity of measures on
  homogeneous spaces.
\newblock {\em Manuscripta Math.}, 94(2):169--186, 1997.

\bibitem[Sou04]{PS}
Ph. Souplet.
\newblock A survey on {$L\sp p\sb \delta$} spaces and their applications to
  nonlinear elliptic and parabolic problems.
\newblock In {\em Nonlinear partial differential equations and their
  applications}, volume~20 of {\em GAKUTO Internat. Ser. Math. Sci. Appl.},
  pages 464--479. Gakk\=otosho, Tokyo, 2004.

\bibitem[Tri11]{Triebel}
Hans Triebel.
\newblock {\em Fractals and spectra}.
\newblock Modern Birkh\"auser Classics. Birkh\"auser Verlag, Basel, 2011.
\newblock Related to Fourier analysis and function spaces.

\bibitem[VK87]{VK87}
A.~L. Vol{\cprime}berg and S.~V. Konyagin.
\newblock On measures with the doubling condition.
\newblock {\em Izv. Akad. Nauk SSSR Ser. Mat.}, 51(3):666--675, 1987.

\bibitem[Wu98]{Wu}
Jang-Mei Wu.
\newblock Hausdorff dimension and doubling measures on metric spaces.
\newblock {\em Proc. Amer. Math. Soc.}, 126(5):1453--1459, 1998.

\end{thebibliography}
\end{document}